\documentclass[12pt,reqno]{amsart}

\usepackage{verbatim}

\newcommand{\alp}{\alpha}
\newcommand{\sig}{\sigma}

\newcommand{\Sig}{\Sigma}

\newcommand{\longc}{,\dotsc,}
\newcommand{\longp}{+\dotsb+}
\newcommand{\longe}{=\dotsb=}
\newcommand{\longle}{\le\dotsb\le}

\newcommand{\seq}{\subseteq}
\newcommand{\stm}{\setminus}

\renewcommand{\>}{\rangle}

\renewcommand{\)}{\right)}

\newcommand{\lfl}{\left\lfloor}
\newcommand{\rfl}{\right\rfloor}

\newcommand{\lfr}{\left\{}
\newcommand{\rfr}{\right\}}

\newtheorem{claim}{Claim}
\newtheorem{theorem}{Theorem}
\newtheorem{corollary}{Corollary}

\newcommand{\refm}[1]{\ref{m:#1}}
\newcommand{\reft}[1]{\ref{t:#1}}
\newcommand{\refs}[1]{\ref{s:#1}}
\newcommand{\refb}[1]{\cite{b:#1}}

\DeclareMathOperator{\ord}{ord}

\title[Zero-sum-free sequences]%
  {Zero-sum-free sequences \\ with few subsequence sums}
\author{Vsevolod F. Lev}
\email{seva@math.haifa.ac.il}
\address{Department of Mathematics, The University of Haifa at Oranim,
  Tivon 36006, Israel}
\makeatletter
\@namedef{subjclassname@2020}{\textup{2020} Mathematics Subject Classification}
\makeatother

\begin{document}
\baselineskip=16pt

\begin{abstract}
Extending the results of Savchev-Chen and Yuan, we show that a
zero-sum-free sequence of length $n$ over an abelian group spans at least
$2n$ distinct subsequence sums, unless it has a simple, explicitly
described structure.
\end{abstract}

\maketitle

\section{Introduction}\label{s:intro}

A sequence of elements of an abelian group is called \emph{zero-sum-free} if
it does not contain a finite, nonempty subsequence with the zero sum of its
terms. The problem of estimating the smallest possible number of distinct
subsequence sums of a finite zero-sum-free sequence was raised by Eggleton
and Erd\H os~\refb{ee}, and since then has attracted much attention; we refer
the reader to the survey by Gao and Geroldinger~\refb{ggS} for a
comprehensive introduction to the subject area and historical overview.
% and the current state of the art.

In one of the early papers on this subject, Olson and White~\refb{ow} have
shown that a zero-sum-free sequence of length $n$ determines at least $2n$
distinct subsequence sums provided that the subgroup generated by the
elements of the sequence is not cyclic; here the empty subsequence with the
zero sum of its terms is counted, too. For the infinite cyclic group
(identified with the group of integers), a complete description of the
length-$n$ sequences with fewer than $2n$ subsequence sums was obtained
in~\cite[Proposition~1]{b:l} under the assumption that all elements of the
sequence are positive; the general case where the elements are allowed to be
negative has never been considered, to the best of our knowledge.

As an immediate corollary of the result of Olson and White, one
recovers~\cite[Theorem~3]{b:ee}: a zero-sum-free sequence generating a
finite, noncyclic group of order $m$ has length $n\le m/2$.

Given a sequence $\alp$ in an abelian group, let $\Sig(\alp)$ denote the set
of all subsequence sums of $\alp$ (including the empty subsequence); thus, if
$\alp=(a_1\longc a_n)$, then
  $$ \textstyle \Sig(\alp) = \lfr \sum_{i\in I} a_i\colon I\seq[1,n] \rfr. $$
In 2007, Savchev and Chen, and independently Yuan, proved a remarkable result
characterizing zero-sum-free sequences of length $n>m/2$ in the finite
\emph{cyclic} group of order $m$.

\begin{theorem}[Savchev-Chen~\refb{sc}, Yuan~\refb{y}]\label{t:scy}
Suppose that $\alp=(a_1\longc a_n)$ is a zero-sum-free sequence of elements
of the cyclic group of order $m$. If $n>m/2$, then, having the elements of
$\alp$ suitably renumbered, there are a group element $a$ of order $m$ and
integers $1=x_1\longle x_n$ satisfying
 $x_{k+1}\le 1+x_1\longp x_k,\ k\in[1,n-1]$, such that $a_k=x_ka$ for all
$k\in[1,n]$. Moreover, $\Sig(\alp)=\{0,a,2a\longc(x_1\longp x_n)a\}$ and
$x_1\longp x_n<m$.
\end{theorem}

By $\ord(a)$ we denote the order of a group element $a$.

In this note we prove the following refinement of the result of
Savchev-Chen-Yuan in the spirit of Olson-White.
\begin{theorem}\label{t:main}
Suppose that $\alp=(a_1\longc a_n)$ is a zero-sum-free sequence of elements
of an abelian group. If $|\Sig(\alp)|<2n$ then, having the elements of $\alp$
suitably renumbered, there are a group element $a$ and integers
 $1=x_1\longle x_n$ satisfying $x_{k+1}\le x_1\longp x_k,\ k\in[1,n-1]$,
such that $a_k=x_ka$ for all $k\in[1,n]$. Moreover,
$\Sig(\alp)=\{0,a,2a\longc(x_1\longp x_n)a\}$, and if $a$ is of finite order,
then $x_1\longp x_n<\ord(a)$.
% As a result, $a_{i+1}\in\Sig(a_1\longc a_i)\ (i\in[1,n-1])$ and
% $\Sig(\alp)=\{ka\colon 0\le k\le x_1\longp x_n\}$.
\end{theorem}

Theorem~\reft{main} is easily seen to imply Theorem~\reft{scy}; the major
difference between the two results is that Theorem~\reft{main} yields the
conclusion in the situation where the set of subsequence sums is small for
whatever reason, not necessarily because ``the whole group is exhausted''.
Loosely speaking, Theorem~\reft{main} shows that forcing $\alp$ to be
structured is not the fact that $\alp$ is ``long'', as compared to the size
of the group, but rather that $\alp$ has few subsequence sums.

The bound $2n$ for the number of subsequence sums is best possible: if $a_1$
and $a_2$ are group elements of sufficiently large order with
$a_2\notin\{-na_1\longc-a_1,0,a_1\longc na_1\}$, and $\alp=(a_1\longc
a_1,a_2)$ where the element $a_1$ is repeated $n-1$ times, then $\alp$ is
zero-sum-free and $|\Sig(\alp)|=2n$, but $\alp$ does not have the structure
described in Theorem~\reft{main}, as it follows by observing that
$\Sig(\alp)$ is not an arithmetic progression.

A finite sequence $\alp$ is \emph{zero-sum} if the sum of its terms is $0$;
it is \emph{minimal} zero-sum if, in addition, none of its nonempty
subsequences is zero-sum. Theorem~\reft{scy} is known
% \refb{sc2}
to have the following equivalent restatement: if $\alp=(a_1\longc a_n)$ is a
minimal zero-sum sequence of elements of the cyclic group of order
$m<2(n-1)$, then there are a group element $a$ and positive integers
$x_1\longc x_n$ such that $x_1\longp x_n=\ord(a)$ and $a_k=x_k a$ for all
$k\in[1,n]$. As a corollary of Theorem~\reft{main}, we obtain the same
conclusion under weaker assumptions.
\begin{corollary}\label{c:the}
Suppose that $\alp=(a_1\longc a_n)$ is a minimal zero-sum sequence of
elements of a finite abelian group. If $|\Sig(\alp)|<2(n-1)$, then there are
a group element $a$ and positive integers $x_1\longc x_n$ such that
$x_1\longp x_n=\ord(a)$ and $a_k=x_k a$ for all $k\in[1,n]$.
\end{corollary}

\begin{proof}
Let $\alp'$ be the sequence of length $n-1$ obtained from $\alp$ by removing
the term $a_n$. Clearly, $\alp'$ is zero-sum-free, and
% $\Sig(\alp)=\Sig(\alp')\cup(\Sig(\alp')+a_i)$ implying
$|\Sig(\alp')|\le|\Sig(\alp)|<2(n-1)$. By Theorem~\reft{main}, there are a
group element $a$ and positive integers $x_1\longc x_{n-1}$ such that
$x_1\longp x_{n-1}<\ord(a)$ and $a_k=x_ka$ for all $k\in[1,n-1]$. Let
$x_n:=\ord(a)-(x_1\longp x_{n-1})$. Then $x_n$ is a positive integer
satisfying $x_1\longp x_n=\ord(a)$, and we have
  $$ x_n a= (\ord(a)-(x_1\longp x_{n-1}))a = -a_1-\dotsb-a_{n-1}=a_n. $$
\end{proof}

% It is not clear to us whether the bound $2(n-1)$ can be substantially
% improved, and whether there is a meaningful version of Corollary~\refc{the}
% for infinite groups.

The proof of Theorem~\reft{main} is presented in the next section; it is a
modified version of the original proof of Savchev and Chen. In view of the
above-mentioned result of Olson-White, it would suffice to prove the theorem
in the case where the underlying group is cyclic. However, we give a
complete, self-contained proof which works for any abelian group, whether it
is finite, cyclic, or neither. An application is considered in the concluding
Section~\refs{app}.

\section{Proof of Theorem~\reft{scy}}

We start with some basic notions and facts about the zero-sum-free sequences.

For a finite sequence $\alp$, we denote by $\sig(\alp)$ the sum of all terms
of $\alp$, and by $\alp_k$ the truncated sequence consisting of the first $k$
terms of $\alp$; here $k$ is a positive integer not exceeding the length of
$\alp$. Thus, writing
 $\alp=(a_1\longc a_n)$,
   $$ \sig(\alp)=a_1\longp a_n
           \ \text{and}\ \alp_k=(a_1\longc a_k),\quad 1\le k\le n. $$

As defined in Section~\refs{intro}, given a sequence $\alp$, by $\Sig(\alp)$
we denote the set of all subsequence sums of $\alp$, including the empty
subsequence.

We notice that if $\alp=(a_1\longc a_n)$ is a zero-sum-free sequence, then
$\Sig(\alp_{n-1})$ is a proper subset of $\Sig(\alp)$; indeed, $\sig(\alp)$
is an element of the later, but not of the former set. It follows that
$|\Sig(\alp)|\ge n+1$ for any zero-sum-free sequence $\alp$ of length $n$.

We say that a nonempty, zero-sum-free sequence $\alp$ is \emph{sharp} if
$\Sig(\alp)=\Sig(\alp_{n-1})\cup\{\sig(\alp)\}$; in other words, $\alp$ is
sharp if it is nonempty, zero-sum-free, and
$|\Sig(\alp)|=|\Sig(\alp_{n-1})|+1$, where $n$ is the length of $\alp$. (This
terminology is different from that used in~\refb{sc}.)

The following claim contains some basic observations, of which the first two
originate from~\refb{sf}, and the third one from \refb{sc} where the proofs
of all three can be found.

For a group element $g$, we denote by $\<g\>$ the cyclic subgroup generated
by $g$; thus, if the order of $g$ is finite, then the subgroup is finite,
too, and $\ord(g)=|\<g\>|$.
\begin{claim}[Smith-Freeze, Savchev-Chen]\label{m:basic}
If $\alp=(a_1\longc a_n)$ is a sharp, zero-sum-free sequence of elements of
an abelian group, then
\begin{itemize}
\item[(a)] $\Sig(\alp_{n-1})$ is a union of the arithmetic progression
    $\{0,a_n,2a_n\longc sa_n\}$, where $1\le s<\ord(a_n)-1$, and a
    (possibly, zero) number of nonzero $\<a_n\>$-cosets;
\item[(b)] $\sig(\alp_{n-1})=sa_n$;
\item[(c)] $a_n$ is the unique group element such that appending it to
    $\alp_{n-1}$ as a last term yields a sharp, zero-sum-free sequence.
    % replacing in $\alp$ the last term
    % $a_n$ with any other group element results in a sequence which is not
    % sharp.
\end{itemize}
\end{claim}

The inequality $s<\ord(a_n)-1$ in (a) should be interpreted as ``either the
order of $a_n$ is infinite, or the order is finite and then the stated
inequality holds''. Other inequalities of this kind are used throughout
without further explanations.

Our next claim is a version of \cite[Lemma~6]{b:sc} extended onto arbitrary
abelian groups (not necessarily cyclic or finite), and with the length
assumption relaxed to the assumption that the subsequence sum set is
small. % The proof is quite different from that in~\refb{sc}.
\begin{claim}\label{m:badstart}
Suppose that $\alp=(a_1\longc a_n)$ is a zero-sum-free sequence with
$|\Sig(\alp)|<2n$. If $|\Sig(\alp_m)|\ge 2m$ for some $m\in[1,n-1]$, then the
terms $a_{m+1}\longc a_n$ of $\alp$ can be permuted (while keeping the order
of the first $m$ terms) so that the resulting sequence is sharp.
\end{claim}

\begin{proof}
We can assume that $m$ is maximal subject to $|\Sig(\alp_m)|\ge 2m$; thus, in
particular, the sequence obtained from $\alp_m$ by appending any of the terms
$a_{m+1}\longc a_n$ is sharp. By Claim~\refm{basic}(c), we have
$a_{m+1}\longe a_n$, and by Claim~\refm{basic}(a), the set $\Sig(\alp_m)$ is
a union of nonzero cosets of the subgroup $\<a_n\>$ and an arithmetic
progression $\{0,a_n,2a_n\longc sa_n\}$ where $1\le s<\ord(a_n)-1$. Therefore
each of $\Sig(\alp_{m+1})\longc\Sig(\alp_{n})$ has the very same structure,
the only difference between these sets being that each time we pass from
$\alp_k$ to $\alp_{k+1}$, the value of $s$ increases by $1$. It follows that
for each $k\in[m+1,n]$, and in particular for $k=n$, the sequence $\alp_k$ is
sharp.
\end{proof}

\begin{claim}\label{m:lastunique}
Suppose that the zero-sum-free sequences $\alp=(a_1\longc a_n)$ and
$\alp'=(a_1'\longc a_n')$ are rearrangements of each other. If both $\alp$
and $\alp'$ are sharp, then $a_n=a_n'$.
\end{claim}

\begin{proof}
By the definition of a sharp sequence, we have
$\Sig(\alp)=\Sig(\alp_{n-1})\cup\{\sig(\alp)\}$ and
$\Sig(\alp')=\Sig(\alp_{n-1}')\cup\{\sig(\alp')\}$ where the unions are
disjoint. Since $\Sig(\alp')=\Sig(\alp)$ and $\sig(\alp')=\sig(\alp)$, we
conclude that $\Sig(\alp_{n-1}')=\Sig(\alp_{n-1})$.

Consider the sequence $\alp'':=(a_1\longc a_{n-1},a_n')$. This sequence is
zero-sum-free in view of  $\Sig(\alp_{n-1})=\Sig(\alp_{n-1}')$, and it is
sharp as
\begin{multline*}
  \big| \big( \Sig(\alp_{n-1}'')+a_n'\big) \stm \Sig(\alp_{n-1}'')\big|
      = \big|\big( \Sig(\alp_{n-1})+a_n'\big) \stm \Sig(\alp_{n-1})\big| \\
       = \big|\big( \Sig(\alp_{n-1}')+a_n'\big) \stm \Sig(\alp_{n-1}')\big|
           = 1.
\end{multline*}
The assertion follows now from Claim~\refm{basic} (c).
\begin{comment}
Let $S:=\Sig(\alp_{n-1})$ and $H:=\<a_n\>$. Since $S=\Sig(\alp_{n-1}')$ and
$a_n$ is a term of $\alp_{n-1}')$, we have $a_n\in S$; also, $0\in S$, while
$-a_n\notin S$ by the zero-sum-free property of $\alp$.

$\stm\Sig(\alp_{n-1}')|=1$, and in view of $\Sig(\alp')=\Sig(\alp)$, this
shows that
  $$ |\Sig(\alp)\stm\Sig(\alp_{n-1}')|=1. $$

Since $\alp$ is sharp, by Claim~\refm{basic}(a), the set
$S:=\Sig(\alp_{n-1})$ is the disjoint union of an arithmetic progression
$\{0,a_n,2a_n\longc sa_n\}\seq \<a_n\>$, where $1\le s<\ord(a_n)-1$, and a
(possibly zero) number of $\<a_n\>$-cosets. Consecutively, $\Sig(\alp)$ is
the disjoint union of the progression
$P:=\{0,a_n,2a_n\longc(s+1)a_n\}\seq\<a_n\>$ and the same $\<a_n\>$-cosets as
those contained in $\Sig(a_{n-1})$.

Let $g$ be the group element such that
$(\Sig(\alp)+a_n')\stm\Sig(\alp)=\{g\}$. Consider the coset $C:=g+\<a_n\>$.
Since $g\notin\Sig(\alp)$, we have $C\not\seq\Sig(\alp)$; on the other hand,
$C$ is not disjoint from $\Sig(\alp)$ since in this case we would have
$\{g\}=(\Sig(\alp)+a_n')\cap C$ which is impossible as there are no
$\<a_n\>$-cosets in which $\Sig(\alp)$ has exactly one element. However, the
only $\<a_n\>$-coset which is neither contained in $\Sig(\alp)$, nor disjoint
from it, is the subgroup $\<a_n\>$ itself. Thus, $C=\<a_n\>$ and
$|(P+a_n')\stm P|=1$ which is possible only if $|\<a_n\>\stm P|=1$,
$a_n'=-a_n$, or $a_n'=a_n$. The first two options are ruled out since $\alp$
is zero-sum-free, and we conclude that $a_n'=a_n$, as wanted.
\end{comment}
\end{proof}

%\begin{proof}
%Since $\alp$ is sharp, by Claim~\refm{basic}(a), the set $\Sig(\alp_{n-1})$
%is the disjoint union of an arithmetic progression $\{0,a_n,2a_n\longc
%sa_n\}\seq \<a_n\>$, where $1\le s<\ord(a_n)-1$, and a (possibly zero) number
%of $\<a_n\>$-cosets. Consecutively, $\Sig(\alp)$ is the disjoint union of the
%progression $P:=\{0,a_n,2a_n\longc(s+1)a_n\}\seq\<a_n\>$ and the same
%$\<a_n\>$-cosets as those contained in $\Sig(a_{n-1})$.
%
%By the definition of a sharp sequence, we have
%$|(\Sig(\alp')+a_n')\stm\Sig(\alp')|=1$, and in view of
%$\Sig(\alp')=\Sig(\alp)$, this shows that
%  $$ |(\Sig(\alp)+a_n')\stm\Sig(\alp)|=1. $$
%
%Let $g$ be the group element such that
%$(\Sig(\alp)+a_n')\stm\Sig(\alp)=\{g\}$. Consider the coset $C:=g+\<a_n\>$.
%Since $g\notin\Sig(\alp)$, we have $C\not\seq\Sig(\alp)$; on the other hand,
%$C$ is not disjoint from $\Sig(\alp)$ since in this case we would have
%$\{g\}=(\Sig(\alp)+a_n')\cap C$ which is impossible as there are no
%$\<a_n\>$-cosets in which $\Sig(\alp)$ has exactly one element. However, the
%only $\<a_n\>$-coset which is neither contained in $\Sig(\alp)$, nor disjoint
%from it, is the subgroup $\<a_n\>$ itself. Thus, $C=\<a_n\>$ and
%$|(P+a_n')\stm P|=1$ which is possible only if $|\<a_n\>\stm P|=1$,
%$a_n'=-a_n$, or $a_n'=a_n$. The first two options are ruled out since $\alp$
%is zero-sum-free, and we conclude that $a_n'=a_n$, as wanted.
%\end{proof}

We say that a sequence $\alp=(a_1\longc a_n)$ of elements of an abelian group
is \emph{connected} if each term of $\alp$ is an algebraic sum of some of the
preceeding terms; that is, if $a_k\in\Sig(\alp_{k-1})-\Sig(\alp_{k-1})$, for
all $k\in[2,n]$. If, indeed, $a_k\in\Sig(\alp_{k-1})$ holds, then we say that
$\alp$ is \emph{strongly} connected.

If $\xi=(x_1\longc x_n)$ is a strongly connected sequence of nonnegative
integers, then $x_k\le\sig(\xi_{k-1})$. If, in addition, $x_1=1$, then from
$\Sig(\xi_k)=\Sig(\xi_{k-1})\cup(\Sig(\xi_{k-1})+x_k)$ it follows that
$\Sig(\xi_k)=[0,\sig(\xi_k)]$, for all $k\in[1,n]$. (In the terminology
of~\refb{sc}, the equality $\Sig(\xi_k)=[0,\sig(\xi_k)]$ means that $\xi$ is
\emph{behaving}.)

\begin{claim}\label{m:goodend}
Every zero-sum-free sequence $\alp$ with $|\Sig(\alp)|<2n$, where $n$ is the
length of $\alp$, can be rearranged so that the resulting sequence is either
connected, or sharp.
\end{claim}

\begin{proof}
Suppose that $\alp=(a_1\longc a_n)$ where the elements are numbered so that
$\alp$ starts with the longest connected sequence possible. If this longest
connected sequence has length $n$, then $\alp$ is connected; suppose thus
that the length is $m-1$, where $m\le n$. Then $a_m$ cannot be represented as
a difference of two elements of $\Sig(\alp_{m-1})$, meaning that the sets
$\Sig(\alp_{m-1})$ and $\Sig(\alp_{m-1})+a_{m}$ are disjoint. Therefore
$|\Sig(\alp_{m})|\ge 2|\Sig(\alp_{m-1})|\ge 2m$. The result now follows from
Claim~\refm{badstart}.
\end{proof}

For an integer set $S$ and an element $a$ of an abelian group, we write
$S\cdot a:=\{sa\colon s\in S\}$.

It is easy to see that if $\alp=(a_1\longc a_n)$ is connected and
zero-sum-free, then for each $k\in[1,n]$ there is an integer $x_k$ such that
$a_k=x_ka_1$ and moreover, $x_1=1$, and if $a_1$ has finite order, then $0\le
x_k<\ord(a_1)$. We say that the sequence of integers $\xi=(x_1\longc x_n)$ is
\emph{associated} with $\alp$. Clearly, in this situation we have
$\Sig(\alp)=\Sig(\xi)\cdot a_1$.

\begin{claim}\label{m:int10}
If $\alp=(a_1\longc a_n)$ is a zero-sum-free, connected sequence of elements
of an abelian group, then the associated integer sequence $\xi$ has positive
terms and is \emph{strongly} connected. Moreover, $\sig(\xi)<\ord(a_1)$.
\end{claim}

\begin{proof}
We use induction on $n$. The case $n=1$ is trivial. Suppose that $n>1$ and
write $\xi=(x_1\longc x_n)$. By the induction hypothesis, $x_1\longc
x_{n-1}\ge 1$, $\sig(\xi_{n-1})<\ord(a_1)$, and $\xi_{n-1}$ is strongly
connected, as a result of which $\Sig(\xi_{n-1})=[0,\sig(\xi_{n-1})]$;
consequently, $\Sig(\alp_{n-1})=[0,\sig(\xi_{n-1})]\cdot a_1$.

Since $\alp$ is connected, we have
  $$ a_n \in \Sig(\alp_{n-1})-\Sig(\alp_{n-1})
                        = [-\sig(\xi_{n-1}),\sig(\xi_{n-1})]\cdot a_1. $$
On the other hand, since $\alp$ is zero-sum-free,
$a_n\notin-\Sig(\alp_{n-1})=[-\sig(\xi_{n-1}),0]\cdot a_1$. It follows that
  $$ x_na_1 = a_n \in [1,\sig(\xi_{n-1})]\cdot a_1. $$
If the order of $a_1$ is infinite, then we immediately conclude that
$x_n\in[1,\sig(\xi_{n-1})]$. If $a_1$ is of finite order, then we arrive at
the same conclusion choosing $t\in[1,\sig(\xi_{n-1})]$ to satisfy
$x_na_1=ta_1$ and observing that $1\le x_n,t<\ord(a_1)$. Recalling that
$\Sig(\xi_{n-1})=[0,\sig(\xi_{n-1})]$, we derive that
$x_n\in\Sig(\xi_{n-1})$; hence $\xi$ is strongly connected and, consequently,
$\Sig(\xi)=[0,\sig(\xi)]$. This shows that every integer from the interval
$[1,\sig(\xi)]$ is the sum of the elements of a subsequence of $\xi$.
Clearly, this subsequence is nonempty. Therefore, every element of the set
$[1,\sig(\xi)]\cdot a_1$ is the sum of the elements of a nonempty subsequence
of $\alp$. Since $\alp$ is zero-sum-free, we conclude that $0\notin
[1,\sig(\xi)]\cdot a_1$, implying the last assertion.
\end{proof}

\begin{claim}\label{m:int2}
If $\xi$ is a strongly connected sequence of positive integers, then so is
the nondecreasing rearrangement of $\xi$.
\end{claim}

\begin{proof}
Let $\xi'$ be the nondecreasing rearrangement of $\xi$. We know that every
term of $\xi$ is the sum of some of the other terms; therefore, every term is
in fact the sum of some strictly smaller terms. It follows that every term of
$\xi'$ is the sum of some of the preceding terms (since \emph{all} terms
smaller than the one under consideration precede it).
% Therefore, each term of $\xi'$
% does not exceed the sum of all the preceding terms, and it follows easily
% that $\Sig(\xi'_k)=[0,\sig(\xi'_k)]$.
\end{proof}

As a direct consequence of Claims~\refm{int10} and~\refm{int2}, we have
\begin{claim}\label{m:connect}
If $\alp=(a_1\longc a_n)$ is a zero-sum-free, connected sequence of elements
of an abelian group, then the conclusion of Theorem~\reft{main} holds true:
having the elements of $\alp$ suitably renumbered, there are a group element
$a$ and integers $1=x_1\longle x_n$ with $x_1\longp x_n<\ord(a)$ such that
$a_k=x_ka$ for all $k\in[1,n]$, and $x_{k+1}\le x_1\longp x_k$ for all
$k\in[1,n-1]$. Moreover, $\Sig(\alp)=[0,x_1\longp x_n]\cdot a$.
\end{claim}

Finally, we can prove our main result.
\begin{proof}[Proof of Theorem~\reft{main}]
If $\alp$ admits a connected rearrangement, then the result follows from
Claim~\refm{connect}; we thus assume that none of the rearrangements of
$\alp$ are connected and show that this assumption leads to a contradiction.

With Claim~\refm{goodend} in mind, we can assume that $\alp$ is sharp. Let
$c:=a_n$, and let $m$ be the multiplicity of $c$ in $\alp$; that is,
$m=|\{i\in[1,n]\colon a_i=c\}|$. Consider a rearrangement of $\alp$ beginning
with the term $c$ repeated $m$ times. Since this rearrangement is not
connected, there exists $k\in[2,n]$ such that $\Sig(\alp_{k-1})$ is disjoint
from $\Sig(\alp_{k-1})+a_k$, resulting in
$|\Sig(\alp_k)|=2|\Sig(\alp_{k-1})|\ge 2k$. By Claim~\refm{badstart} we can
further rearrange the terms of $\alp$, starting from the $(k+1)$-th term, so
that the resulting rearrangement $\alp'$ is sharp. By
Claim~\refm{lastunique}, the last term of $\alp'$ is then equal to $c$,
contradicting the way $\alp'$ is constructed.
\end{proof}

\section{High-multiplicity elements in zero-sum-free sequences with few
  subset sums}\label{s:app}

The Savchev-Chen-Yuan theorem has numerous applications. Using
Theorem~\reft{main} instead, one can improve or generalize some of the
corresponding results. We consider just one example.

Suppose that $\alp$ is a zero-sum-free sequence of length $n$ over the cyclic
group of order $m$. Bovey, Erd\H os, and Niven~\refb{ben} have shown that if
$n>\frac12\,m$, then $\alp$ contains a term of multiplicity at least
$2n-m+1$; as remarked in~\refb{ben}, this estimate is best possible whenever
$\frac23\,(m-1)\le n<m$. An improvement for the complementary range
$\frac12\,m<n\le\frac23\,(m-1)$ was obtained by Gao and
Geroldinger~\refb{gg}, and then Savchev and Chen~\refb{sc} applied their
result to prove an estimate which is sharp for all values of $n>\frac12\,m$;
namely, the maximum multiplicity is at least
  $$ \begin{cases}
         n-\lfl\frac{m-1}3\rfl, &\quad \frac12\,m<n\le\frac23\,(m-1), \\
         2n-m+1, &\quad \frac23\,(m-1)\le n<m.
     \end{cases} $$

Using Theorem~\reft{main} instead of Theorem~\reft{scy}, we prove
\begin{theorem}\label{t:mult}
Suppose that $\alp$ is a zero-sum-free sequence of length $n$ over an abelian
group. If $\alp$ has at most $m<2n$ distinct subsequence sums, then there is
a group element that appears in $\alp$ with the multiplicity at least
  $$ \max_{r\ge 1} \frac2r\(n-\frac{m-1}{r+1}\). $$
\end{theorem}
Since a zero-sum-free sequence in the cyclic group of order $m$ can have at
most $m$ distinct subsequence sums, the case $r=1$ gives the
 Boven-Erd\H os-Niven estimate, and taking $r=2$ we recover the Savchev-Chen
bound. Substituting $r\ge 3$ does not lead to any further improvement, which
is very expectable bearing in mind that the Savchev-Chen bound is tight.
Notice, however, that Theorem~\reft{mult} is applicable under substantially
weaker assumptions.

\begin{proof}[Proof of Theorem~\reft{mult}]
Applying Theorem~\reft{main}, and having $\alp$ appropriately ordered, we
write $\alp=(a_1\longc a_n)$ and $\xi=(x_1\longc x_n)$ where $1=x_1\longle
x_n$ are the integers appearing in the conclusion of Theorem~\reft{main}. For
each $j\ge 1$, denote by $\nu_j$ be the number of indices $i\in[1,n]$ with
$x_i=j$. We have $\nu_1+\nu_2+\nu_3+\dotsb=n$ and
$\nu_1+2\nu_2+3\nu_3+\dotsb=\sig(\xi)=|\Sig(\alp)|-1\le m-1$. Therefore,
denoting by $\mu$ the largest multiplicity of an element of $\alp$, and
observing that $\mu=\max_{j\ge 1}\nu_j$, for any integer $r\ge 1$ we obtain
\begin{align*}
  \frac12\,r(r+1)\mu + m
     &\ge (r\nu_1+(r-1)\nu_2\longp \nu_r) \\
     &\qquad\qquad      + (1+\nu_1+2\nu_2\longp r\nu_r+(r+1)\nu_{r+1}+\dotsb)  \\
     & \ge (r+1) n+1
\end{align*}
which is equivalent to $\mu\ge\frac2r\(n-\frac{m-1}{r+1}\)$.
\end{proof}

\bigskip

\end{document}